\documentclass[letter,11pt]{amsart}\usepackage{amsmath}

\usepackage[english]{babel}
\usepackage{amsfonts,lmodern,fancyhdr,lastpage,graphicx,nccfoots,caption}
\usepackage{amssymb, mathrsfs, verbatim}
\usepackage{bm}
\usepackage{graphicx}
\usepackage{subfigure}
\usepackage{enumitem}
\usepackage{amsthm}
\usepackage{color}

\newtheorem{theorem}{Theorem}[section]
\newtheorem{proposition}[theorem]{Proposition}
\newtheorem{definition}[theorem]{Definition}
\newtheorem{remark}[theorem]{Remark}

\newtheorem{corollary}[theorem]{Corollary}
\newtheorem{example}[theorem]{Example}
\newtheorem*{maintheorem}{Main Result}

\definecolor{Green}{rgb}{0.010,0.7,0.02}

\begin{document}

\date{\today }
\title[Embeddings for the space $LD_\gamma^{p}$]{ Embeddings for the space $%
LD_\gamma^{p}$ on sets of finite perimeter}
\author[N. Chemetov]{ Nikolai V. Chemetov}
\address{University of Lisbon, Lisbon, Portugal}
\email{nvchemetov@fc.ul.pt, nvchemetov@gmail.com}
\author[A. Mazzucato]{Anna L. Mazzucato}
\address{Penn State University, University Park, PA, 16802, U.S.A.}
\email{alm24@psu.edu}

\begin{abstract}
Given an open set with finite perimeter $\Omega\subset \mathbb{R}^n$, we
consider the space $LD_\gamma^{p}(\Omega)$, $1\leq p<\infty$, of functions
with $p$th-integrable deformation tensor on $\Omega$ and with $p$%
th-integrable trace value on the essential boundary of $\Omega$. We
establish the continuous embedding $LD_\gamma^{p}(\Omega)\subset
L^{pN/(N-1)}(\Omega)$. The space $LD_\gamma^{p}(\Omega)$ and this embedding
arise naturally in studying the motion of rigid bodies in a viscous,
incompressible fluid.
\end{abstract}

\maketitle

\section{Introduction}

\label{ms}

In this work, we establish Sobolev-type embeddings for a non-standard
function space that arise in the study of the motion of rigid bodies
in a viscous, incompressible fluid.

The problem of the motion of solid bodies in a viscous fluid filling a
bounded container, has been studied by several authors --- we cite, in
particular, Hoffmann, Starovoitov \cite{HOST}, San Mart\'{\i}n, Starovoitov,
Tucsnak \cite{SST}, Feireisl, Hillairet, Ne\v{c}asov\'{a} \cite{FHN}, Bost,
Cottet, Maitre \cite{BCM}, Gunzburger, Lee, Seregin \cite{GLS}, Takahashi 
\cite{T}, Judakov \cite{Yu}. The authors of these works considered the
no-slip condition, that is, the velocity is assumed to equal the velocity at
the surface of the rigid bodies and the velocity of the container walls,
also assumed rigid. In the simplest case in which the container is fixed, the
velocity vanishes at the walls.

However, it has been shown mathematically that this assumption leads to some
non-physical results, namely, under the no-slip condition collision between
the bodies and between the bodies and the walls cannot happen in finite time
(see Hesla \cite{HES}, Hillairet \cite{HIL}, Starovoitov \cite{STA2} among
others).

One way to include the possibility of collision that is physically motivated
is to allow slippage at the boundaries. There are several ways to allow for
a non-trivial slip at the boundary by modifying the boundary condition. The
Navier boundary condition models slip with friction and it is amenable to a
theoretical analysis. The first to study collisions under the Navier
conditions were Neustupa and Penel \cite{NP1,NP2}, who considered a
prescribed collision of a ball with a flat wall, while the free motion of a
single rigid body in the whole space $\mathbb{R}^{3}$ was investigated in 
\cite{PS111}. The local existence result, i.e., up to the time of first
collision, for motion of a solid in the presence of walls and slip was
recently obtained by G\'{e}rard-Varet and Hillairet \cite{GH2}. In \cite{GHC}, 
it was shown that, when Navier boundary conditions are imposed on both the
solid and walls, collision of the solid body with the boundary indeed can
happen in finite time. The work \cite{cnn} contains a global existence
result for weak solutions when the Navier friction condition is imposed at the surface of the
body and the no-slip condition is imposed at the container walls.

When bodies collide, the fluid domain, which coincide with the portion of
the container that is exterior to the bodies can have low regularity,
typically at the level of cusps, especially if the solid bodies have smooth
boundary (see Figure \ref{fig2}). In this situation, both Poincar\'e and
Korn's inequalities do not hold in general, so no standard embedding results
are available.

However, in studying existence of weak solution for the fluid-interaction
problem past collision, one is confronted with the integrability of
functions that have only bounded deformation tensor in $L^2$. Our main
result is a Sobolev-type embedding result for functions with this level of
regularity in cusp domains, and even rougher sets, more precisely, sets of
finite perimeter, if in addition some information is available on the trace
of the function at the boundary.

%the space of $L^{2}$ integrable bounded deformations in domains with bad
%regularity of boundary.
\begin{figure}[tbp]
%\begin{center}
\subfigure[Body touching the container walls.]{
\scalebox{0.1}{\includegraphics{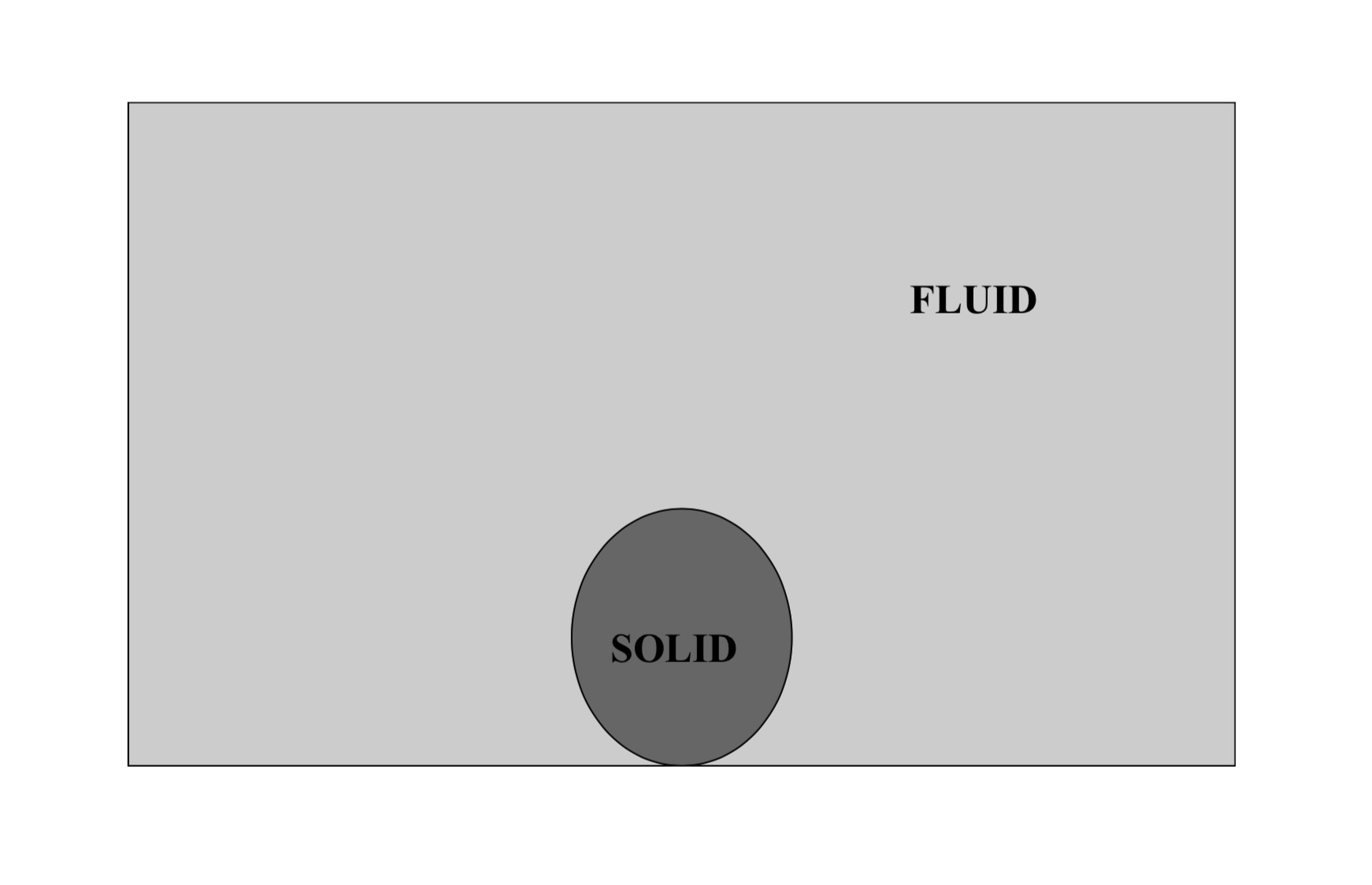}}
\label{fig2.a}} \vspace*{0.2in} \hspace{\subfigtopskip}\hspace{%
\subfigbottomskip} %\caption{Body touching the container walls}
\subfigure[Cuspidal subregion generated after
touching.]{
\scalebox{0.12}{\includegraphics{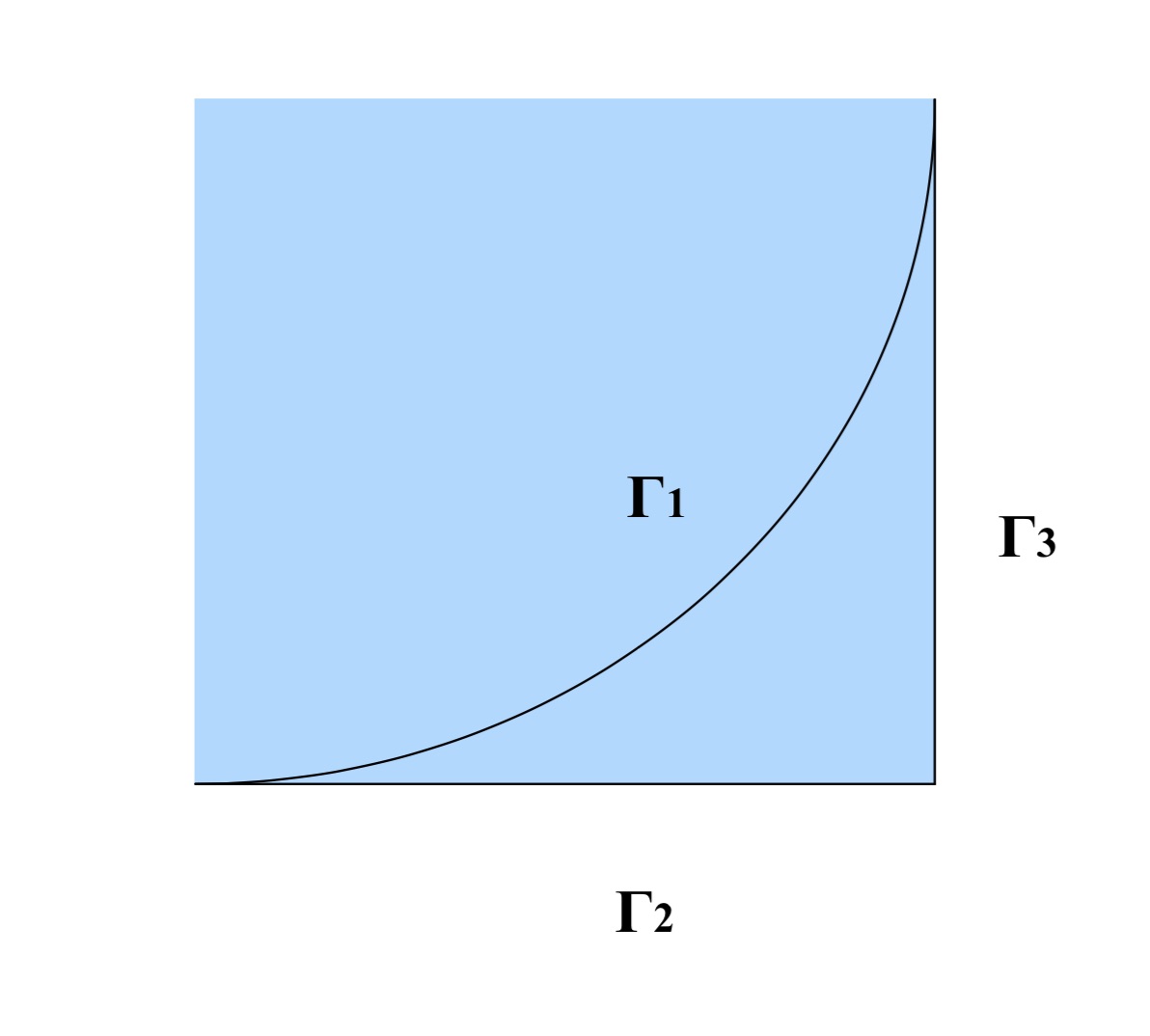}}
\label{fig2.b}} %\end{center}
\caption{The fluid geometry in the presence of collisions.}
\label{fig2}
\end{figure}

%\smallskip

%\bigskip

We begin by recalling the notion of bounded deformation. Let $\Omega \subset 
\mathbb{R}^{N}$\ be a bounded open set. In applications, $N=2, 3$. We
consider a vector-function $\ \mathbf{v:x}\in \Omega \rightarrow \mathbb{R}%
^{N}$, and define the tensor of deformation $\mathbb{D}\mathbf{v}=\frac{1}{2}%
(\nabla \mathbf{v}+\left( \nabla \mathbf{v}\right) ^{T})$ with components 
\begin{equation*}
d{_{ij}(}\mathbf{v})=\frac{1}{2}\left( \frac{\partial v_{i}}{\partial x_{j}}+%
\frac{\partial v_{j}}{\partial x_{i}}\right) ,\qquad i,j=1,...N.
\end{equation*}

\begin{definition}
\label{def2.5} Given $p\geq 1$, we define the space of functions with $L^{p}$%
-bounded deformation as 
\begin{equation*}
LD^{p}(\Omega )=\left\{ \mathbf{v}\in L^{p}(\Omega ):\,\mathbb{D}\mathbf{v}%
\in L^{p}(\Omega )\right\} ,
\end{equation*}%
endowed by the norm $\ ||\mathbf{v}||_{LD^{p}(\Omega )}=||\mathbf{v}%
||_{L^{p}(\Omega )}+||\mathbb{D}\mathbf{v}||_{L^{p}(\Omega )}.$
\end{definition}

\bigskip

%Since at the moment of collision of the rigid body with the boundary of the
%domain and/or of collision of two bodies, the fluid will occupy a domain
%with cusps. In the following considerations we will be interested in
%embedding results involving the space $LD^{2}(\Omega )$ for $\Omega $ with
%the cusps. The reason is that under 
%the study of the \ the motion of
%colliding rigid bodies it appears the important question of 
Our main motivation is the validation of the convective term in the Navier -
Stokes equations for the fluid-structure interaction problem in the presence
of collisions 
\begin{equation}  \label{eq:convective}
\int_{\Omega }(\mathbf{u}\otimes \mathbf{u}):\mathbb{D}\boldsymbol{\psi }\,d%
\mathbf{x\qquad }\text{for a test function\quad\ }\boldsymbol{\psi }\in
LD^{2}(\Omega )
\end{equation}%
in cuspidal domains $\Omega $ \ (see the definition 2.1 in \cite{cnn}). The
convective term is then well defined as a distribution, if the solution $%
\mathbf{u}\in LD^{2}(\Omega )$ belongs at least to $L^{4}(\Omega )$.

We briefly review existing embedding results for domains with cusps. There
are well known embedding results involving the Sobolev space $%
W_{2}^{1}(\Omega )$ when $\Omega $ has cusps (see \cite{ad}, \cite{Gr}, \cite{MP}). \ However, the methods used in these works can not be applied in the
case of $LD^{2}(\Omega )$ if one wants a bound on the norm in $LD^{2}(\Omega
)$. \ The optimal embedding theorem $\,W_{2}^{1}(V(x^{\alpha
}))\hookrightarrow L^{r}(V(x^{\alpha }))\,$ for $r\in \lbrack 1,\frac{%
2(\alpha +1)}{\alpha -1}]\,$ in the cuspidal domain \ 
\begin{equation*}
V(x^{\alpha })=\left\{ \mathbf{x}=(x,y):~0<x<1,\quad 0<y<x^{\alpha }\right\}
\subset \mathbb{R}^{2}
\end{equation*}%
was obtained in \cite{MM1}. The embedding result $\,W_{2}^{1}(V(x^{\alpha
}))\hookrightarrow L^{q}(\partial V(x^{\alpha })),\,$ with $\,\,1\leqslant
q\leqslant 2\,$ for optimal values of $\alpha :\,\,\alpha <1+\frac{2}{q},\,$
was proved in \cite{Acosta2}. For a more complete description of optimal
embedding results in cuspidal domains, we refer to \cite{Besov}, \cite{Kil}, 
\cite{Lab}, \cite{MP} and \cite{W}.

We will show next with an example that knowing $\mathbf{u}\in LD^2(\Omega)$
is \emph{not} enough to guarantee that the convective term in Equation %
\eqref{eq:convective} is well defined and additional hypotheses are needed.

\begin{example}
\label{ex1} We consider the cusp domain $V(x^{2})$. This type of cusp
appears at the moment in which a solid disk collides with a flat walls in
two space dimensions. We take the vector function 
\begin{equation}
\mathbf{w}_s=((s-1)yx^{-s},x^{1-s})  \label{ex_w}
\end{equation}%
with $s$ a real parameter to be chosen later on. One can compute the
associated deformation tensor 
\begin{equation*}
{\mathbb{D}}\mathbf{w}_s=\left[ 
\begin{array}{cc}
-s(s-1)yx^{-s-1} & 0 \\ 
0 & 0%
\end{array}%
\right]
\end{equation*}%
as in \cite{Acosta3}, p. 219-221, from which it follows that 
\begin{equation*}
||{\mathbb{D}}\mathbf{w}_s||_{L^{2}(V(x^{2}))}^{2} \leq
C\int_{0}^{1}x^{6-2(s+1)}dx,
\end{equation*}
for some positive constant $C$. Similarly, given $q\geq 1$, we calculate 
\begin{eqnarray*}
||\mathbf{w}_s||_{L^{q}(V(x^{2}))}^{q} &\simeq &C\int_{0}^{1}\left(
\int_{0}^{x^{2}}(y^{q}x^{-qs}+x^{-q(s-1)}) \, dy\right) \, dx \simeq
C\int_{0}^{1}x^{-q(s-1)+2} \, dx.
\end{eqnarray*}%
Consequently, 
\begin{equation*}
\mathbf{w}_s\in LD^{2}(V(x^{2}))\qquad \text{for any}\quad s<1+\frac{3}{2}.
\end{equation*}%
Taking $q=2+\varepsilon $, $\varepsilon>0$, we conclude at the same time
that 
\begin{equation*}
\mathbf{w}_s\notin L^{2+\varepsilon }(V(x^{2}))\qquad \text{for}\quad s=1+%
\frac{3}{2+\varepsilon }.
\end{equation*}%
In this case, in particular, we cannot make sense of the convective term in
Equation \eqref{eq:convective}.
\end{example}

For the applications we have in mind, additional information is available on
the integrability of the trace of the function at the boundary. The theory
of sets of finite perimeter, which covers most cusp domains, provides a
suitable framework for defining the trace on rough (non-Lipschitz)
boundaries. 
%In what follows we will discuss the behavior of boundary values of $LD^{2}-$%
%functions. To do it let us present some notations, definitions and results
%given in 

We informally define the space $LD_\gamma^{p}(\Omega)$, $1\leq p<\infty$, of
functions with $p$th-integrable deformation tensor on $\Omega$ and with $p$%
th-integrable trace value on the essential boundary of $\Omega$ (see
Definition \ref{LD_2_b}). Then, our main result is the following embedding
result.

\begin{maintheorem}
Let $\Omega \subset \mathbb{R}^N$ be a bounded open set with finite
perimeter. Then, there is a continuous embedding 
\begin{equation*}
LD_{\gamma }^{p}(\Omega ) \hookrightarrow L^{\frac{pN}{(N-1)}}(\Omega ).
\end{equation*}
\end{maintheorem}

In the case $p=2$, $N=2$, we therefore have that $\mathbf{u} \in L^4(\Omega
) $, as required to define the convective term of Equation %
\eqref{eq:convective}.

The paper is organized as follows. In Section \ref{sec:prelim}, we recall
principal facts about sets of finite perimeters and functions of bounded
deformation. We also discuss a few preliminary results needed in the proof of our
main theorem, which is presented in Section \ref{sec:main}. Throughout, we
use standard notation for classical spaces, such as the Sobolev spaces $%
W^{k,p}(\Omega )$.

\subsection*{Acknowledgments}

The second author was partially supported by the US National Science
Foundation grant DMS-1615457. The second author acknowledges the hospitality
of the Department of Mathematics at the University of Lisbon, where part of
this work was conducted.

\section{Sets of finite perimeters and bounded deformation}

\label{sec:prelim}

In this section, we briefly recall the needed results on sets of finite
perimeter, and states key properties of functions of bounded deformation on
such sets. We will use the notation and results in \cite{Ambr}, \cite{Leoni}%
, \cite{Volpert} and \cite{Volpert2}.

We denote the $d$-dimensional Lebesgue measure by $\mathcal{L}^{d}$ and the
Hausdorff measure for $d$-dimensional sets in $\mathbb{R}^{N}$, $N\geq d$,
as $\mathcal{H}^{d} $. We also denote the space of functions of bounded
variation in $\mathbb{R}^N$ by $BV(\mathbb{R}^{N})$.

\begin{definition}
\label{def3} Let $E\subset \mathbb{R}^{N}$ be a bounded $\mathcal{L}^{N}-$%
measurable set. We denote the characteristic function of the set $E$ by $%
\chi _{E}$. \ If $\chi _{E}\in BV(\mathbb{R}^{N}),$ then $E$ is called a set
with finite perimeter. The finite positive number 
\begin{equation*}
P(E)=|\nabla \chi _{E}|(\mathbb{R}^{N})=\sup_{\mathbf{\phi }}\left\{
\int_{E} div \, \mathbf{\phi }\,d\mathbf{x}:\quad |\mathbf{\phi }|\leq
1,\quad \mathbf{\phi }\in C_{c}^{1}(\mathbb{R}^{N})\right\}
\end{equation*}
is called the perimeter of the set $E$.
\end{definition}

In the definition above, we have used that, if $\chi _{E}\in BV(\mathbb{R}%
^{N}),$ then the generalized gradient $\ \nabla \chi _{E}=(\mu _{1},...,\mu
_{N})=\mathbf{\mu }$\ \ is a vector with components given by bounded Radon
measures $\mu _{i},$ $\ i=1,...N,$ satisfying 
\begin{equation*}
\int_{E} div \, \mathbf{\phi }\,d\mathbf{x}=-\int_{\mathbb{R}^{N} }(\mathbf{%
\phi },d\mathbf{\mu )},\qquad \text{for any }\mathbf{\phi }=(\phi
_{1},...,\phi _{N})\in C_{c}^{1}(\mathbb{R}^{N}).
\end{equation*}

The following results are discussed in \cite[pages 154-156]{Volpert} and in 
\cite[page 159, Proposition 3.62]{Ambr}, for instance.

\begin{proposition}
\label{Prop} The following holds

\begin{enumerate}[label=(\arabic*)]

\item The set of all sets with finite perimeter forms algebra, that is, if $%
E, F$ have finite perimeter then the sets $\mathbb{R}^{N}\backslash E, \,
E\cup F,$ $E\cap F$ also have finite perimeter.

\item If $E$ is a Lipschitz domain, then $E$ is a set with finite perimeter
and $P(E)=\mathcal{H}^{N-1}(\partial E).$
\end{enumerate}
\end{proposition}

\begin{remark}
If we consider the motion of finitely-many rigid bodies in a solid
container, all with Lipschitz boundaries, an hypothesis which covers
situation of physical and computational interest, then the domain occupied
by the fluid will be a set of finite perimeter at all times by Proposition %
\ref{Prop}. (See Figure \ref{figure_anna}.)
\end{remark}

\begin{figure}[h]
\scalebox{0.4}{\includegraphics{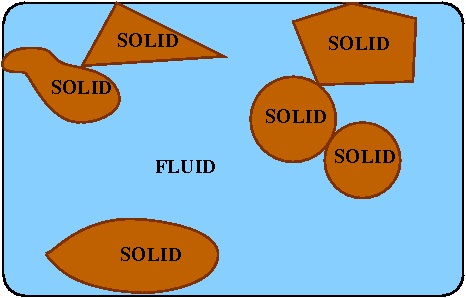}}
\caption{Solid bodies immersed in a fluid in a rigid container.}
\label{figure_anna}
\end{figure}

\bigskip

We introduce the concept of \emph{essential boundary} for a measurable set
(see e.g. \cite[pages 256, 258]{Volpert}, \cite[page 158]{AMT}). We shall
denote the ball of radius $\rho >0$ and center $\mathbf{x}\in \mathbb{R}^{N}$
by $B_{\rho }(\mathbf{x})$, and its volume by $\omega _{N}(\rho )=\mathcal{L}%
^{N}(B_{\rho }(\mathbf{x}))$. We shall also define the unit sphere 
\begin{equation*}
\mathbb{S}^{N-1}=\{\mathbf{a}\in \mathbb{R}^{N}:\;\;|\mathbf{a}|=1\}
\end{equation*}%
and the hyperplane 
\begin{equation*}
P_{\mathbf{a}}=\{\mathbf{y}\in \mathbb{R}^{N}:\mathbf{y}\cdot \mathbf{a}=0\}
\end{equation*}%
through the origin with normal vector $\mathbf{a}$ in $\mathbb{S}^{N-1}.$ We
shall also need the line $l_{\mathbf{a}}(\mathbf{x})$ with direction vector $%
\mathbf{a}\in \mathbb{S}^{N-1}$ through $\mathbf{x}\in P_{\mathbf{a}}.$

\begin{definition}
\label{def5} Let $E$ be a given measurable subset of $\mathbb{R}^{N}.$ A
point $\mathbf{x\in } E$ is point of density (respectively rarefaction) of
the set $E$ if 
\begin{equation*}
\lim_{\rho \rightarrow 0^{+}}\frac{\mathcal{L}^{N}(E\cap B_{\rho }(\mathbf{x}%
))}{\omega _{N}(\rho )}=1 \quad (\text{respectively} = 0).
\end{equation*}
We denote by $E_{\ast }$ the set of all points of density of $E$ and by $%
E^{\ast}$ the complement of the set of points of rarefaction of $E.$ The set
\ $\partial ^{\ast }E=E^{\ast }\backslash E_{\ast }$ is then called the
essential boundary of $E$.
\end{definition}

\bigskip

We next recall some facts about the essential boundary for sets with finite
perimeter. \ For more details we refer the reader to \cite{AMT}, \cite{EG}, 
\cite{Fe}, \cite{Gi}, \cite{Volpert2}, and \cite{Zi}.

\begin{proposition}
\label{Prop2} Let $E$ be a set with finite perimeter and let $\partial
^{\ast }E$ be its essential boundary. Then

\begin{enumerate}[label=(\roman*), ref=(\roman*)]

\item The boundary $\partial ^{\ast }E$ is countably $\mathcal{H}^{N-1}$%
-rectifiable, that is, 
\begin{equation*}
\partial ^{\ast }E=\cup _{n=1}^{\infty }K_{n}\cup S,
\end{equation*}%
where $\mathcal{H}^{N-1}(S)=0$, the sets $K_n$ are pairwise disjoint, and
each $K_n$ is a compact subset of a $C^{1}$ hypersurface in $\mathbb{R}^{N}$ 
\begin{equation*}
K_{n}=\Phi _{n}(A_{n})
\end{equation*}
for a compact subset $A_{n}\subset \mathbb{R}^{N-1}$ and a $C^1$ map $\Phi_n
: \mathbb{R}^{N-1} \rightarrow \mathbb{R}^{N}$ \cite[page 205]{EG}.

\item the unit normal $\bm{\nu }=\bm{\nu }(\mathbf{x})$ exists\ for $%
\mathcal{H}^{N-1}-$a.a. points $\mathbf{x}\in \partial ^{\ast }E$ \, \cite[%
page 205]{EG}, \cite[pages 227-228]{Volpert2}, \cite[pages 154, 158]{AMT}.

\item \label{Prop2.item3} Let $\mathbf{a}\in \mathbb{S}^{N-1}$. For $%
\mathcal{L}^{N-1}$-a.a. $\mathbf{x}\in P_{\mathbf{a}}$, the set $l_{\mathbf{a%
}}(\mathbf{x})\cap E_{\ast }$ is the union of a finite number of open
intervals with disjoint closures, and the union of the boundary points of
the intervals\ coincides with the set $l_{\mathbf{a}}(\mathbf{x})\cap
\partial ^{\ast }E$ \, \cite[page 233]{Volpert2}.
\end{enumerate}
\end{proposition}

In what follows, we denote by $BD(\Omega )$ the space of functions with
bounded deformation on $\Omega$ in analogy with $BV(\Omega)$. For $p\geq 1$
the space $LD^{p}(\Omega )$ is a subspace $BD(\Omega )$. Hence, we can apply
the result of \cite{Temam,RT1}, showing that the trace of functions
in $LD^{p}(\Omega )$ is well defined. (The same result was also described
carefully in \cite[Proposition 4.1]{bab}.)

\begin{proposition}
\label{approx_trace} Let $\Omega $ be a set with finite perimeter and let $%
\partial ^{\ast }\Omega $ be its essential boundary. \ Let $\mathbf{u}(%
\mathbf{x})\in LD^{p}(\Omega )$. Then for $\mathcal{H}^{N-1}-$a.e. $\mathbf{x%
}\in \partial ^{\ast }\Omega $, there exist a vector $\gamma \mathbf{u}(%
\mathbf{x})\in \mathbb{R}^{N}$, such that%
\begin{equation}
\lim_{\rho \rightarrow 0^{+}}\frac{2}{\omega _{N}(\rho )}\int_{B_{\rho} (%
\mathbf{x},\bm{\nu }) \cap \Omega}|\mathbf{u}(\mathbf{y})-\gamma \mathbf{u}(%
\mathbf{x})|\,d\mathbf{y}=0  \label{apr}
\end{equation}%
where $\bm{\nu }=\bm{\nu }(\mathbf{x})\in \mathbb{S}^{N-1}$ is the internal
normal at $\mathbf{x}\in \partial ^{\ast }\Omega $ \ and the half ball $%
B_{\rho }(\mathbf{x},\bm{\nu })$ is defined as%
\begin{equation}
B_{\rho }(\mathbf{x},\bm{\nu })=\{\mathbf{y}\in \mathbb{R}^{N}:\quad |%
\mathbf{y}-\mathbf{x}|<\rho ,\quad (\mathbf{y}-\mathbf{x})\cdot \bm{\nu }%
\,>0\}.  \label{b}
\end{equation}
\end{proposition}

The assignment $\mathbf{x} \mapsto \gamma \mathbf{u}(\mathbf{x})$ defines a
trace map on the essential boundary of $\Omega$ for elements of $%
LD^p(\Omega) $.

We study next the integrability properties of the trace map. If $\Omega$ is
a Lipschitz domain, then the embedding $LD^{2}(\Omega )\hookrightarrow
L^{2}(\partial \Omega )$ can be established using the same approach of
Theorem 1.1 in \cite[ page 117]{Temam} and Theorem 3.2 in \cite{bab} (see
also the theorem and example given on pages 224-227 of \cite{Volpert}). For
cusp domains, the above result is generally not true, as the function
introduced in Example \ref{ex1} shows.

\begin{example}
\label{ex2} %As we have shown in this example \ref{ex1}, 
We consider again the function $\mathbf{w}_s$ defined in Equation %
\eqref{ex_w} on the cusp domain $V(x^2)$. This function belongs to $%
LD^{2}(V(x^{2}))$\ for any given $s<1+\frac{3}{2}.$ We then fix 
\begin{equation*}
s\in \lbrack \frac{3}{2},1+\frac{3}{2}),
\end{equation*}%
and calculate the $L^2$ norm of $\mathbf{w}_s$ along the part of the
boundary defined by $0<x<1,\quad y=0$ 
\begin{equation*}
\int_{0}^{1}x^{-2(s-1)}dx=+\infty ,
\end{equation*}%
that is, $\mathbf{w}_s\notin L^{2}(\partial V(x^{2}))$. We conclude that the
inclusion \ 
\begin{equation*}
\qquad LD^{2}(\Omega )\hookrightarrow L^{2}(\partial \Omega )
\end{equation*}%
is not valid if $\Omega$ has cusps.
\end{example}

The above example also shows, however, that there can be functions in $%
LD^2(\Omega)$ with square-integrable trace on the essential boundary, even
if $\Omega$ is rough (take $\mathbf{w}_s$ with $s$ small enough). This
observation justifies the introduction of the following space.

\begin{definition}
\label{LD_2_b} Let $\Omega \subset \mathbb{R}^{N}$ be a bounded open set
with finite perimeter.\ Let $LD_{\gamma }^{p}(\Omega )$ be the space of
functions $\mathbf{u}\in LD^{p}(\Omega )$ such that the trace of $\mathbf{u}$%
, $\gamma \mathbf{u}$, defined by Equation \eqref{apr}, is $p$%
th-integrable on the essential boundary $\partial ^{\ast }\Omega$ with
respect of Hausdorff measure $\mathcal{H}^{N-1}$. \ The space $LD_{\gamma
}^{p}(\Omega )$ is a normed vector space with norm 
\begin{equation*}
\left\Vert \mathbf{u}\right\Vert _{LD_{\gamma }^{p}(\Omega
)}^{p}=\int_{\Omega }|\mathbb{D}\mathbf{u|}^{p}\,d\mathbf{x}+\int_{\partial
^{\ast }\Omega }|\gamma \mathbf{u|}^{p}\,d\mathcal{H}^{N-1}(\mathbf{x}).
\end{equation*}
\end{definition}

\begin{remark}
The space $LD_{\gamma }^{2}(\Omega )$ appears naturally in the construction
of the weak solutions to the problem of motion of rigid bodies in a viscous
fluid under Navier slip conditions. We refer to the \emph{a priori} estimate
(2.8) established in Theorem 2.1 of \cite{cnn} (see also the \emph{a priori}
estimate (4.5) in Theorem 1 of \cite{GH2}).
\end{remark}

Our main result is based on an extension of the Fundamental Theorem of
Calculus on sections, valid for functions with integrable variation in $%
W^{1,p}$, to functions with integrable deformation in $LD^{p}$. To this end,
we introduce the following notation. For given $\mathbf{a}\in \mathbb{S}%
^{N-1}$ \ we define the section 
\begin{equation*}
\Omega _{\mathbf{y},\mathbf{a}}=\left\{ t\in \mathbb{R}:\mathbf{y}+t\mathbf{a%
}\,\in \Omega \right\}
\end{equation*}%
of $\Omega $ corresponding to a point $\mathbf{y}\in P_{\mathbf{a}}(\subset 
\mathbb{R}^{N-1}).$ \ If $\Omega _{\mathbf{y},\mathbf{a}}$ is empty, we set%
\begin{equation*}
\int_{\Omega _{\mathbf{y},\mathbf{a}}}f(\mathbf{y}+t\mathbf{a})\ dt=0
\end{equation*}%
for any $\mathcal{L}^{N}-$integrable function $f:\Omega \rightarrow \mathbb{R%
}$. Then the Fubini-Tonelli theorem implies that 
\begin{equation}
\int_{\Omega }f(\mathbf{x})\,d\mathbf{x}=\int_{P_{\mathbf{a}}}\left(
\int_{\Omega _{\mathbf{y},\mathbf{a}}}f(\mathbf{y}+t\mathbf{a})\ dt\right) d%
\mathcal{L}^{N-1}(\mathbf{y)}.  \label{z16}
\end{equation}

The following result, showing absolute continuity on lines, is an analogue
of Theorems 7.13 and 10.35 in \cite{Leoni}.
\begin{proposition}
\label{theorem0} Let $\Omega \subset \mathbb{R}^{N}$ \ be an open set. Let $%
\mathbf{a}_{k}\in \mathbb{S}^{N-1},$ \ $k=1,...,N,$ be arbitrary linearly
independent vectors.

For given $\mathbf{u}\in LD^{p}(\Omega )$ there exists a representative $%
\overline{\mathbf{u}}$ of $\mathbf{u},$ such for each\quad $k=1,...,N$ and $%
\mathcal{L}^{N-1}$- a.e. $\mathbf{y}\in \mathbb{R}^{N-1}$ the function 
\begin{equation*}
v_{k}(t)=\mathbf{a}_{k}\,\cdot \overline{\mathbf{u}}(\mathbf{y}+t\mathbf{a}%
_{k})
\end{equation*}%
is absolutely continuous on $t\in \Omega _{\mathbf{y},k}=\Omega _{\mathbf{y},%
\mathbf{a}_{k}}$\ and\ the following formula%
\begin{equation}
v_{k}(t)=v_{k}(t^{\prime })+\int_{t^{\prime }}^{t}\mathbf{a}_{k}\mathbb{D}%
\mathbf{u}(\mathbf{y}+s\mathbf{a}_{k})\cdot \mathbf{a}_{k}\,\ ds  \label{abs}
\end{equation}%
is valid for any $[t^{\prime },t]\subset \Omega _{\mathbf{y},k}.$
\end{proposition}

\begin{proof}
We consider a sequence of standard mollifiers $\{\varphi _{\varepsilon
}\}_{\varepsilon >0}$ (see C.4, pages 552-560, of \cite{Leoni}), and for
every $\varepsilon >0$ define 
\begin{equation*}
\mathbf{u}^{\varepsilon }=\mathbf{u}\ast \varphi _{\varepsilon }\qquad \text{%
on \quad }\Omega _{\varepsilon }=\left\{ \mathbf{x}\in \Omega :~dist(\mathbf{%
x},\partial \Omega )>\varepsilon \right\} .
\end{equation*}%
By the same approach as in Lemma 10.16 of \cite{Leoni}, we have%
\begin{equation*}
\lim_{\varepsilon \rightarrow 0^{+}}\int_{\Omega _{\varepsilon }}\left( |%
\mathbf{u}^{\varepsilon }-\mathbf{u}|^{p}\,+|\mathbb{D}\mathbf{u}%
^{\varepsilon }-\mathbb{D}\mathbf{u}|^{p}\right) \,d\mathbf{x}=0.
\end{equation*}%
Using \eqref{z16}, it follows that 
\begin{eqnarray*}
\int_{P_{k}}\left( \int_{\Omega _{\mathbf{y},k}}|\mathbb{D}\mathbf{u}(%
\mathbf{y}+t\mathbf{a}_{k})|^{p}\,dt\right) \ d\mathcal{L}^{N-1}(\mathbf{y)}
&<&\infty , \\
\lim_{\varepsilon \rightarrow 0^{+}}\int_{P_{k}}\left( \int_{(\Omega
_{\varepsilon })_{\mathbf{y},k}}|\mathbb{D}\mathbf{u}^{\varepsilon }(\mathbf{%
y}+t\mathbf{a}_{k})-\mathbb{D}\mathbf{u}(\mathbf{y}+t\mathbf{a}%
_{k})|^{p}\,dt\right) \ d\mathcal{L}^{N-1}(\mathbf{y)} &=&0.
\end{eqnarray*}%
Therefore, there exists a subsequence $\{\varepsilon _{n}\},$ such that for $%
\mathcal{L}^{N-1}-$a.a. $\mathbf{y}\in P_{k}$, $k=1,...,N,$ \ we have 
\begin{eqnarray}
\int_{\Omega _{\mathbf{y},k}}|\mathbb{D}\mathbf{u}(\mathbf{y}+t\mathbf{a}%
_{k})|^{p}\,dt &<&\infty ,  \notag \\
\lim_{\varepsilon _{n}\rightarrow 0^{+}}\int_{(\Omega _{\varepsilon _{n}})_{%
\mathbf{y},k}}|\mathbb{D}\mathbf{u}^{\varepsilon _{n}}(\mathbf{y}+t\mathbf{a}%
_{k})-\mathbb{D}\mathbf{u}(\mathbf{y}+t\mathbf{a}_{k})|^{p}\,dt &=&0.
\label{z17}
\end{eqnarray}

If we set $\mathbf{u}^{_{n}}=\mathbf{u}^{\varepsilon _{n}},$ then the
sequence $\left\{ \mathbf{u}_{n}\right\} $ converges point-wise to $\mathbf{u%
}$\ \ for $\mathcal{L}^{N}-$ a.a. points of $\Omega $ by Theorem C.19 and
Corollary B.122 of \cite{Leoni}. Therefore, the set 
\begin{equation*}
E=\{\mathbf{x}\in \Omega :~\lim_{n\rightarrow \infty }\mathbf{u}^{_{n}}(%
\mathbf{x})\text{ \ exists in \ }\mathbb{R}^{N}\}
\end{equation*}%
is well defined and such that $\mathcal{L}^{N}(\Omega \backslash E)=0.$ We let 
\begin{equation*}
\overline{\mathbf{u}}(\mathbf{x})=\left\{ 
\begin{array}{ll}
\lim_{n\rightarrow \infty }\mathbf{u}^{_{n}}(\mathbf{x}), & \text{if }%
\mathbf{x}\in E; \\ 
0, & \text{if }\mathbf{x}\in \Omega \backslash E%
\end{array}%
\right.
\end{equation*}%
and the function $\overline{\mathbf{u}}$ is one of the representatives the
equivalence class of $\mathbf{u}$.\ \ Fubini's theorem implies that%
\begin{equation*}
\int_{P_{k}}\mathcal{L}^{1}(\left\{ t\in \mathbb{R}:\quad \mathbf{y}+t%
\mathbf{a}_{k}\notin E\,\right\} \cap \Omega _{\mathbf{y},k})\,\ d\mathcal{L}%
^{N-1}(\mathbf{y)}=0.
\end{equation*}%
Thus for $\mathcal{L}^{N-1}-$a.a. $\mathbf{y}\in P_{k}$, 
\begin{equation}
\mathbf{y}+t\mathbf{a}_{k}\in E\quad \text{for}\quad \mathcal{L}^{1}-\text{%
a.a. \ }t\in \Omega _{\mathbf{y},k},\text{ \quad }\forall k=1,...,N.
\label{z00}
\end{equation}

We denote a generic $N$-dimensional rectangle with the edges parallel to the
vectors $\mathbf{a}_{1},...,\mathbf{a}_{N}$ in $\mathbb{R}^{N}$ by%
\begin{equation*}
\mathbb{A}=\left\{ \mathbf{x}=\sum_{i,j=1}^{N}t_{k}\,\mathbf{a}_{k}:\quad
t_{k}\in \lbrack c_{k},d_{k}]\subset \mathbb{R},\text{\quad }\forall
k=1,...,N\right\} .
\end{equation*}%
We take rectangles $\mathbb{A}\subset \Omega $ \ with $c_{k},d_{k}$ all
rational numbers. For $\varepsilon >0$ sufficiently small, $\mathbb{A}%
\subset \Omega _{\varepsilon },$ so that by \eqref{z17} we have that for $\ 
\mathcal{L}^{N-1}-$a.a. \ $\mathbf{y}\in P_{k}$ 
\begin{eqnarray}
\int_{c_{k}}^{d_{k}}|\mathbb{D}\mathbf{u}(\mathbf{y}+t\mathbf{a}%
_{k})|^{p}\,dt &<&\infty ,  \notag \\
\lim_{n\rightarrow \infty }\int_{c_{k}}^{d_{k}}|\mathbb{D}\mathbf{u}^{n}(%
\mathbf{y}+t\mathbf{a}_{k})-\mathbb{D}\mathbf{u}(\mathbf{y}+t\mathbf{a}%
_{k})|^{p}\ dt &=&0.  \label{z19}
\end{eqnarray}%
\ 

For each $k=1,...,N$ and$\ \mathcal{L}^{N-1}-$a.a. \ $\mathbf{y}\in P_{k},$\
we define%
\begin{equation*}
v_{k}^{n}(t)=\mathbf{u}^{n}(\mathbf{y}+t\mathbf{a}_{k})\cdot \mathbf{a}%
_{k},\ \text{ \quad }t\in \lbrack c_{k},d_{k}].
\end{equation*}%
Using \eqref{z00} we can choose $t^{\prime }\in \lbrack c_{k},d_{k}]$ such
that $\mathbf{y}+t^{\prime }\mathbf{a}_{k}\,\in E.$\ \ Then, the following
limit exists 
\begin{equation}
v_{k}^{n}(t^{\prime })\underset{n\rightarrow \infty }{\longrightarrow }%
v_{k}(t^{\prime })=\overline{\mathbf{u}}(\mathbf{y}+t^{\prime }\mathbf{a}%
_{k})\cdot \mathbf{a}_{k}.  \label{z18}
\end{equation}%
Since $v_{k}^{n}\in C^{\infty }([c_{k},d_{k}]),$ we have%
\begin{eqnarray*}
v_{k}^{n}(t) &=&v_{k}^{n}(t^{\prime })+\int_{t^{\prime }}^{t}\frac{d}{ds}%
\left( v_{k}^{n}(s)\right) \ ds \\
&=&v_{k}^{n}(t^{\prime })+\int_{t^{\prime
}}^{t}\sum_{i,j=1}^{N}a_{i}a_{j}d_{ij}(\mathbf{u}^{n}(\mathbf{y}+s\mathbf{a}%
_{k}))\ ds\qquad \text{for all}\quad t\in \lbrack c_{k},d_{k}].
\end{eqnarray*}%
Hence, \eqref{z19}-\eqref{z18} imply the existence of the limit 
\begin{equation*}
\lim_{n\rightarrow \infty }v_{k}^{n}(t)=\overline{\mathbf{u}}(\mathbf{y}%
+t^{\prime }\mathbf{a}_{k})\cdot \mathbf{a}_{k}+\int_{t^{\prime }}^{t}%
\mathbf{a}_{k}\mathbb{D}\mathbf{u}(\mathbf{y}+s\mathbf{a}_{k}\,)\cdot 
\mathbf{a}_{k}\ ds,\quad \quad \forall t\in \lbrack c_{k},d_{k}].
\end{equation*}%
The definition of $E$ and $\overline{\mathbf{u}}$ give then that, for each $%
k=1,...,N$ \ and$\ \ \mathcal{L}^{N-1}-$a.a. \ $\mathbf{y}\in P_{k},$ 
\begin{equation}
\left\{ \mathbf{y}+t\mathbf{a}_{k}\,:\quad \forall t\in \lbrack
c_{k},d_{k}]\right\} \subset E  \label{z20}
\end{equation}%
(compare with \eqref{z00}) and the functions $v_{k}(t)=\mathbf{a}_{k}\,\cdot 
\overline{\mathbf{u}}(\mathbf{y}+t\mathbf{a}_{k}\,)$ \ satisfy%
\begin{equation}
v_{k}(t)=v_{k}(t^{\prime }\mathbf{)}+\int_{t^{\prime }}^{t}\mathbf{a}_{k}%
\mathbb{D}(\mathbf{u}(\mathbf{y}+s\mathbf{a}_{k}))\cdot \mathbf{a}%
_{k}\,ds,\quad \quad \forall t\in \lbrack c_{k},d_{k}].  \label{z21}
\end{equation}%
Consequently, each function $v_{k}\mathbf{=}v_{k}(t)$ is absolutely
continuous on $[c_{k},d_{k}]$ and $\ $%
\begin{equation*}
\frac{dv}{dt}(t)=\mathbf{a}_{k}\mathbb{D}(\mathbf{u}(\mathbf{y}+t\mathbf{a}%
_{k}))\cdot \mathbf{a}_{k}\qquad \text{for \ }\mathcal{L}^{1}\text{ - a.e.}%
\quad t\in \lbrack c_{k},d_{k}],
\end{equation*}%
which can be shown as in Lemma 3.31 of \cite{Leoni}.

Now, if $\widetilde{\mathbb{A}}\subset \Omega $ is another such rectangle
with the property that 
\begin{equation*}
\lbrack c_{k},d_{k}]\cap \lbrack \widetilde{c}_{k},\widetilde{d}_{k}]\neq
\varnothing ,\text{\quad }\forall k=1,...,N,
\end{equation*}%
then, by taking $\mathbf{y}\in P_{k}$, which is admissible for both $\mathbb{%
A}$ and $\widetilde{\mathbb{A}}$, \ and $\ t^{\prime }\in \lbrack
c_{k},d_{k}]\cup \lbrack \widetilde{c}_{k},\widetilde{d}_{k}],$ it follows
that from \eqref{z20} and \eqref{z21} that $v_{k}$ is absolutely continuous
on the interval $[c_{k},d_{k}]\cup \lbrack \widetilde{c}_{k},\widetilde{d}%
_{k}].$

Since $\Omega $ can be written as a countable union of closed rectangles of
this type and since the union of countably many sets of $\mathcal{L}^{N-1}$%
-measure zero still has $\mathcal{L}^{N-1}$-measure zero, using \eqref{z20}, %
\eqref{z21} we conclude that for $\mathcal{L}^{N-1}$-a.e. $\mathbf{y}\in
P_{k}$, \ the function $v_{k}(t)$ \ is absolutely continuous on any
connected component of $\Omega _{\mathbf{y},k}$.
\end{proof}

\bigskip

\bigskip

\bigskip

Next, we formulate and prove a result concerning a non-tangential approach
to characterize the trace. In what follows, we let $\Omega $ be an open
set of finite perimeter. We recall that $\bm{\nu }=\bm{\nu }(\mathbf{x})\in 
\mathbb{S}^{N-1}$ denotes the internal normal at $\mathbf{x}\in \partial
^{\ast }\Omega $\ and $B_{1}(\mathbf{x},\bm{\nu })$ is the half ball with
radius equal 1, defined by\ \eqref{b}.

Let $\mathbf{a}\in \mathbb{S}^{N-1}\cap B_{1}(\mathbf{x},\bm{\nu })$\ be
arbitrary fixed vector and set 
\begin{equation}
\lambda _{\mathbf{a}}(E)=\mathcal{L}^{N-1}(\pi _{\mathbf{a}}E)
\label{lambda}
\end{equation}%
for any measurable (Borel) set $E\subset \mathbb{R}^{N}$, where $\pi _{%
\mathbf{a}}E$ is the projection of the set $E\subset \mathbb{R}^{N}$ onto
the plane $P_{\mathbf{a}}$. (We refer to \cite[pages 235-236]{Volpert2} for a
discussion of properties of the Borel measure $\lambda _{\mathbf{a}}$.)

\begin{proposition}
\label{direction_trace} Let $\mathbf{u}\in LD_{\gamma }^{p}(\Omega )$. The
following limit exists 
\begin{equation}
\gamma _{\mathbf{a}}\mathbf{u}(\mathbf{x})=\lim_{\varepsilon \rightarrow
0^{+}}\frac{1}{\varepsilon }\int_{0}^{\varepsilon }\mathbf{u}(\mathbf{x}+s%
\mathbf{a})\,ds\qquad \text{for\quad }\lambda _{\mathbf{a}}-a.e.\quad 
\mathbf{x}\in \partial ^{\ast }\Omega ,  \label{eq:direction_trace}
\end{equation}%
such that%
\begin{equation}
\gamma _{\mathbf{a}}\mathbf{u}(\mathbf{x})=\gamma \mathbf{u}(\mathbf{x}%
)\qquad \text{for\quad }\lambda _{\mathbf{a}}-a.e.\quad \mathbf{x}\in
\partial ^{\ast }\Omega .
\end{equation}
\end{proposition}

We omit the proof Proposition \ref{direction_trace}, as it is essentially
the same as that of Theorem in Section 11.2, pages 243-245, of \cite%
{Volpert2}. In fact the proof of this above-mentioned theorem relies on the
structure of the set with finite perimeter and the existence of the trace
values $\gamma \mathbf{u}$ for a given function $\mathbf{u}.$ In our case,
when $\mathbf{u}\in LD_{\gamma }^{p}(\Omega )$, \ the existence of $\gamma 
\mathbf{u}$ is guaranteed by Proposition \ref{approx_trace}.

\bigskip

\begin{corollary}
\label{cor}\bigskip Let $\Omega \subset \mathbb{R}^{N}$ \ be an open set of
finite perimeter. Then under the assumptions and the notation of
Proposition \ref{theorem0}, Formula \eqref{abs} is valid for any $t^{\prime
},t\in \mathbb{R},$ such that%
\begin{equation*}
\lbrack t^{\prime },t]\subset \overline{\Omega _{\mathbf{y},k}}.
\end{equation*}
\end{corollary}

\begin{proof}
By Proposition \ref{theorem0} we have that for $\mathcal{L}^{N-1}$- a.e. $%
\mathbf{y}\in P_{\mathbf{a}_{k}},$%
\begin{equation}
v_{k}(t)=v_{k}(t^{\prime })+\int_{t^{\prime }}^{t}\mathbf{a}_{k}\mathbb{D}%
\mathbf{u}(\mathbf{y}+s\mathbf{a}_{k})\cdot \mathbf{a}_{k}\ ds\qquad \text{%
for any \ }[t^{\prime },t]\subset \Omega _{\mathbf{y},k}\text{.}
\label{labs}
\end{equation}%
Moreover, there exists a finite number of disjoint intervals $[t_{0}^{\prime
},t_{0}]$, such that%
\begin{equation*}
t_{0}^{\prime }\neq t_{0},\text{\qquad }[t_{0}^{\prime },t_{0}]\subset 
\overline{\Omega _{\mathbf{y},k}}\text{\qquad and\qquad }\mathbf{y}%
+t_{0}^{\prime }\mathbf{a}_{k},\text{ }\mathbf{y}+t_{0}\mathbf{a}_{k}\in
\partial ^{\ast }\Omega 
\end{equation*}%
by  Proposition \ref{Prop2} \ref{Prop2.item3}.

If we integrate \eqref{labs} over $t^{\prime }\in (t_{0}^{\prime
},,t_{0}^{\prime }+\varepsilon ),$ divide by $\varepsilon $, and take the
limit $\varepsilon \rightarrow 0,$ then Proposition \ref{direction_trace}
implies that \eqref{labs} is valid for $t^{\prime }=t_{0}^{\prime }.$ By the
same way we can demonstrate the validity of \eqref{labs} for the point $%
t=t_{0}$. Combining these identifications with the integral representation
for $v_{k}$ above, we obtain Formula \eqref{abs} for all $[t^{\prime },t]$
in $\overline{\Omega _{\mathbf{y},k}}$.
\end{proof}

\section{Proof of the main result}

\label{sec:main}

We are now ready to prove our main result. We first recall a needed proposition. 
For a vector $\boldsymbol{\xi }=(\xi _{1},\dots ,\xi _{N})\in \mathbb{R}^{N}$
\ and each $i\in \left\{ 1,\ldots ,N\right\} $, we introduce the vectors 
\begin{equation}
\widehat{\boldsymbol{\xi }}_{i}=(\xi _{1},\dots ,\xi _{i-1},\xi
_{i+1},...,\xi _{N})\in \mathbb{R}^{N-1}.  \label{ksi}
\end{equation}%

For a proof of the following proposition, we refer to \cite{Temam}, pages 128-129,
Lemma 1.1.
\begin{proposition}
\label{l:productDecomp} Let $\theta _{i}=\theta _{i}(\widehat{\boldsymbol{%
\xi }}_{i})$ be non-negative integrable functions in $\mathbb{R}^{N-1}$, $%
i=1,\ldots ,N$. Then,
\begin{equation}
\int_{\mathbb{R}^{N}}\left( \prod_{i=1}^{N}\theta _{i}\right) ^{\frac{1}{N-1}%
}\,d\mathbf{\xi }\leq \prod_{i=1}^{N}\left( \int_{\mathbb{R}^{N-1}}\theta
_{i}(\widehat{\boldsymbol{\xi }}_{i})\,d\widehat{\boldsymbol{\xi }}%
_{i}\right) ^{\frac{1}{N-1}}.  \label{r3}
\end{equation}
\end{proposition}

The main result of this work is the following theorem.

\begin{theorem}
\label{t:maintheorem} Let $\Omega $ be a bounded open set with finite
perimeter. If $\mathbf{u}\in LD_{\gamma }^{p}(\Omega )$, then $\mathbf{u}\in
L^{\frac{pN}{(N-1)}}(\Omega )$ and there exists a positive constant $C$,
depending only on $N$, $p$ and the diameter of the domain $\Omega $, such
that%
\begin{equation}
\left\Vert \mathbf{u}\right\Vert _{L^{\frac{pN}{(N-1)}}(\Omega )}\leq
C\left\Vert \mathbf{u}\right\Vert _{LD_{\gamma }^{p}(\Omega )}.
\label{e:BDembedding}
\end{equation}
\end{theorem}

\begin{proof}
We follow closely the proof in Theorem 1.2, page 117, of \cite{Temam} and
Theorem 6.95, pages 333-336, of \cite{DD}, and combine them with the ideas
developed in the theorem of Section 5, pages 218-220, in \cite{Volpert}. We
adapt this approach to the case at hand of sets with finite perimeter.

\textbf{(I)} As a warm-up for the general case, we start by considering two
space dimensions and $p=2$. Let $\left\{ \mathbf{e}_{1},\mathbf{e}%
_{2}\right\} $ be the Euclidean basis of $\mathbb{R}^{2}.$ We denote a point
in $\mathbb{R}^{2}$ with $\mathbf{x}=(x_{1},x_{2})$ and a vector field on $%
\mathbb{R}^{2}$ with $\mathbf{u}=(u,v)$.

\textit{Step 1:} Since the set $\Omega $ has a finite perimeter, then by
Part \ref{Prop2.item3} of Proposition \ref{Prop2} for $\mathcal{L}^{1}$-a.a. 
$x_{2}\in \mathbb{R}$, the intersection 
\begin{equation*}
\Omega (x_{2})=l_{\mathbf{e}_{1}}((0,x_{2}))\cap \Omega
\end{equation*}%
consists of a finite number $M_{2}(x_{2})$ of open intervals with disjoint
closures 
\begin{equation*}
\Omega (x_{2})=\cup _{l=1}^{M_{2}(x_{2})}\triangle _{2,l},\qquad \text{such
that}\quad \overline{\triangle _{2,l}}\cap \overline{\triangle _{2,m}}%
=\varnothing ,\quad \forall l\neq m,
\end{equation*}%
where $\triangle _{2,l}=(\mathbf{c}_{2,l}(x_{2}),\mathbf{d}_{2,l}(x_{2}))$
is a straight line connecting the points 
\begin{equation*}
\mathbf{c}_{2,l}(x_{2})=(c_{2,l}(x_{2}),x_{2}),\mathbf{\ \ d}%
_{2,l}(x_{2})=(d_{2,l}(x_{2}),x_{2})\in \partial ^{\ast }\Omega .
\end{equation*}%
Consequently, Corollary \ref{cor} implies that for such admissible $x_{2}\in 
\mathbb{R}$ and arbitrary chosen $x_{1}\in \Omega (x_{2})$, there exists an
index $k\in \left\{ 1,\ldots ,M_{2}(x_{2})\right\} $, such that $x_{1}\in
\triangle _{2,k}$ and 
\begin{equation*}
u\mathbf{(x)}=u(x_{1},x_{2})=\gamma u(\mathbf{c}_{2,k}(x_{2}))+%
\int_{c_{2,k}(x_{2})}^{x_{1}}\partial _{x_{1}}u(s,x_{2})\ ds.
\end{equation*}%
It follows that 
\begin{align}
u^{2}\mathbf{(x)}& \leq 2\left[ |\gamma u(\mathbf{c}%
_{2,k}(x_{2}))|^{2}+(d_{2,k}(x_{2})-c_{2,k}(x_{2}))\cdot \right.  \notag \\
& \left. \qquad \qquad \qquad \qquad \qquad \cdot
\int_{c_{2,k}(x_{2})}^{d_{2,k}(x_{2})}|\partial _{x_{1}}u(s,x_{2})|^{2}\ ds 
\right]  \notag \\
& \leq C\sum\limits_{l=1}^{M_{2}(x_{2})}\left[ |\gamma \mathbf{u}(\mathbf{c}%
_{2,l}(x_{2}))|^{2}+\int_{c_{2,l}(x_{2})}^{d_{2,l}(x_{2})}|d_{11}(\mathbf{u}%
)(s,x_{2})|^{2}\ ds\right]  \notag \\
& =f_{2}(x_{2})  \label{e:uSquareBound}
\end{align}%
where the constant $C$ depends only on the diameter of $\Omega .$

In the same fashion, for $\mathcal{L}^{1}$-a.a. $x_{1}\in \mathbb{R}$ \ the
intersection 
\begin{equation*}
\Omega (x_{1})=l_{\mathbf{e}_{2}}((x_{1},0))\cap \Omega
\end{equation*}%
consists of a finite number $M_{1}(x_{1})$ of open intervals with disjoint
closures 
\begin{equation*}
\Omega (x_{1})=\cup _{l=1}^{M_{1}(x_{1})}\triangle _{1,l},\qquad \text{such
that}\quad \overline{\triangle _{1,l}}\cap \overline{\triangle _{1,m}}%
=\varnothing ,\quad \forall l\neq m,
\end{equation*}%
where $\triangle _{1,l}=(\mathbf{c}_{1,l}(x_{1}),\mathbf{d}_{1,l}(x_{1}))$
is a straight line connecting the points 
\begin{equation*}
\mathbf{c}_{1,l}(x_{1})=(x_{1},c_{1,l}(x_{1})),\mathbf{\ \ d}%
_{1,l}(x_{1})=(x_{1},d_{1,l}(x_{1}))\in \partial ^{\ast }\Omega .
\end{equation*}%
For admissible $x_{1}\in \mathbb{R}$ and arbitrary chosen $x_{2}\in \Omega
(x_{1})$, there exists an index $k\in \left\{ 1,\ldots ,M_{1}(x_{1})\right\}
,$ such that $x_{2}\in \triangle _{1,k}$ and 
\begin{equation*}
v\mathbf{(x)}=v(x_{1},x_{2})=\gamma v(\mathbf{c}_{1,k}(x_{1}))+%
\int_{c_{1,k}(x_{1})}^{x_{2}}\partial _{x_{2}}v(x_{1},s)\ ds.
\end{equation*}%
Hence 
\begin{eqnarray}
v^{2}\mathbf{(x)} &\leq &C\sum\limits_{l=1}^{M_{1}(x_{1})}\left[ |\gamma 
\mathbf{u}(\mathbf{c}_{1,l}(x_{1}))|^{2}+%
\int_{c_{1,l}(x_{1})}^{d_{1,l}(x_{1})}|d_{22}(\mathbf{u})(x_{1},s)|^{2}\ ds%
\right]  \notag \\
&=&f_{1}(x_{1}).  \label{e:vSquareBound}
\end{eqnarray}

Multiplying \eqref{e:uSquareBound} with \eqref{e:vSquareBound} and
integrating over $\Omega ,$ by Proposition \ref{l:productDecomp} (or simply
by the Fubini-Tonelli Theorem) we obtain%
\begin{equation*}
\int_{\Omega }u^{2}\mathbf{(x)}v^{2}\mathbf{(x)}\ d\mathbf{x}\leq
\int_{\Omega }f_{1}(x_{1})f_{2}(x_{2})\ d\mathbf{x}\leq
\int_{I_{1}}f_{1}(x_{1})\ dx_{1}\int_{I_{2}}f_{2}(x_{2})\ dx_{2},
\end{equation*}%
where $I_{i},$ $i=1,2,$ are the projections of $\Omega $ onto the $x_{i}$%
-coordinate axis. We have 
\begin{equation*}
\left. 
\begin{array}{c}
\displaystyle{\int_{I_{1}}\sum\limits_{l=1}^{M_{1}(x_{1})}|\gamma \mathbf{u}(%
\mathbf{c}_{1,l}(x_{1}))|^{2}\ dx_{1}} \\ 
\\ 
\displaystyle{\int_{I_{2}}\sum\limits_{l=1}^{M_{2}(x_{2})}|\gamma \mathbf{u}(%
\mathbf{c}_{2,l}(x_{2}))|^{2}\ dx_{2}}%
\end{array}%
\right\} \leq \int_{\partial ^{\ast }\Omega }|\gamma (\mathbf{u})\mathbf{|}%
^{2}\,d\mathcal{H}^{N-1}(\mathbf{x})
\end{equation*}%
by the properties of the measure $\lambda _{\mathbf{a}}$ given on the pages
235-236, section 7, of \cite{Volpert2}. Therefore, 
\begin{align}
\int_{\Omega }u^{2}\mathbf{(x)}v^{2}\mathbf{(x)}\,d\mathbf{x}& \leq
C\,\int_{I_{2}}f_{1}(x_{1})\ dx_{1}\int_{I_{1}}f_{2}(x_{2})\ dx_{2}  \notag
\\
& \leq C\left( \int_{\partial ^{\ast }\Omega }|\gamma (\mathbf{u})\mathbf{|}%
^{2}\,d\mathcal{H}^{N-1}(\mathbf{x})+\int_{\Omega }|\mathbb{D}\mathbf{u|}%
^{2}\,d\mathbf{x}\right) ^{2}=C\left\Vert \mathbf{u}\right\Vert _{LD_{\gamma
}^{2}(\Omega )}^{4}.  \label{e:CrossProductBound}
\end{align}

\textit{2nd step:} Now we consider the basis \ $\mathbf{a}_{1}=\frac{1}{%
\sqrt{2}}(1,1),$\ \ $\mathbf{a}_{2}=\frac{1}{\sqrt{2}}(-1,1).$ We denote the
coordinates of $\mathbf{x}$\ in the basis $(\mathbf{a}_{1},\mathbf{a}_{2})$
by $(\xi _{1},\xi _{2}),$ that is, 
\begin{equation*}
\mathbf{x}=(\mathbf{x}_{1},\mathbf{x}_{2})=\xi _{1}\mathbf{a}_{1}+\xi _{2}%
\mathbf{a}_{2}.
\end{equation*}

Again, for $\mathcal{L}^{1}$- a.e. $\mathbf{y}_{2}=\xi _{2}\mathbf{a}_{2}\in
P_{\mathbf{a}_{1}},$ i.e., for $\mathcal{L}^{1}$- a.e. $\xi _{2}\in \mathbb{R%
}$, the intersection of lines parallel to $\mathbf{a}_{1}$ with the domain $%
\Omega ,$ passing through $\mathbf{y}_{2},$%
\begin{equation*}
\Omega (\xi _{2})=l_{\mathbf{a}_{1}}(\mathbf{y}_{2})\cap \Omega
\end{equation*}%
consists of a finite number $M_{2}(\xi _{2})$ of open intervals with
disjoint closures, such that for$\ \mathbf{x}\in \Omega (\xi _{2}),$\ there
exists an interval%
\begin{equation*}
(\mathbf{c}_{2,k}(\xi _{2}),\mathbf{d}_{2,k}(\xi _{2}))\subset \Omega (\xi
_{2})\qquad \text{with\quad }\mathbf{c}_{2,k}(\xi _{2}),\mathbf{d}_{2,k}(\xi
_{2})\in \partial ^{\ast }\Omega .
\end{equation*}%
For simplicity of notation, we assume that this interval, being a part of $%
l_{\mathbf{a}_{1}}(\mathbf{y}_{2}),$ is described as 
\begin{equation*}
(\mathbf{c}_{2,k}(\xi _{2}),\mathbf{d}_{2,k}(\xi _{2}))=\left\{ \,\mathbf{y}=%
\mathbf{y}_{2}+s\mathbf{a}_{1}\in \mathbb{R}^{N}:~s\in (c_{2,k}(\xi
_{2}),d_{2,k}(\xi _{2}))\right\} .
\end{equation*}

By applying Corollary \ref{cor} to the function 
\begin{equation*}
v_{2}(\xi _{1},\xi _{2})=\mathbf{a}_{1}\,\cdot \mathbf{u}(\mathbf{y}_{2}+\xi
_{1}\mathbf{a}_{1})
\end{equation*}%
and proceeding as in \eqref{e:uSquareBound}-\eqref{e:vSquareBound}, we obtain%
\begin{align}
v_{2}^{2}& \leq 2\left[ |\mathbf{a}_{1}\,\cdot \gamma \mathbf{u}(\mathbf{c}%
_{2,k}(\xi _{2}))|^{2}\right.  \notag \\
& \qquad \left. +(d_{2,k}(\xi _{2})-c_{2,k}(\xi _{2}))\int_{c_{2,k}(\xi
_{2})}^{d_{2,k}(\xi _{2})}|\mathbf{a}_{2}\mathbb{D}\mathbf{u}(\mathbf{y}%
_{2}+s\mathbf{a}_{1})\cdot \mathbf{a}_{2}|^{2}\ ds\right]  \notag \\
& \leq \sum\limits_{l=1}^{M_{2}(\xi _{2})}\left[ |\gamma \mathbf{u}(\mathbf{c%
}_{2,l}(\xi _{2}))|^{2}+\int_{c_{2,l}(\xi _{2})}^{d_{2,l}(\xi _{2})}|\mathbb{%
D}\mathbf{u}(\mathbf{y}_{2}+s\mathbf{a}_{1})|^{2}\ ds\right]  \notag \\
& =f_{2}(\xi _{2}).  \label{v2}
\end{align}%
Similarly, for $\mathcal{L}^{1}$- a.e. $\mathbf{y}_{1}=\xi _{1}\mathbf{a}%
_{1}\in P_{\mathbf{a}_{2}},$ that is, or $\mathcal{L}^{1}$- a.e. $\xi
_{1}\in \mathbb{R}$, \ the intersection of the line parallel to $\mathbf{a}%
_{2}$ with $\Omega $ passing through $\mathbf{y}_{1}$ 
\begin{equation*}
\Omega (\xi _{1})=l_{\mathbf{a}_{2}}(\mathbf{y}_{1})\cap \Omega
\end{equation*}%
is a finite number $M_{1}(\xi _{1})$ of open intervals with disjoint
closures. Then for$\ \mathbf{x}\in \Omega (\xi _{1})$ there exists an
interval, such that 
\begin{equation*}
(\mathbf{c}_{1,k}(\xi _{1}),\mathbf{d}_{1,k}(\xi _{1}))\subset \Omega (\xi
_{1})\qquad \text{with\quad }\mathbf{c}_{1,k}(\xi _{1}),\mathbf{d}_{1,k}(\xi
_{1})\in \partial ^{\ast }\Omega .
\end{equation*}%
Defining 
\begin{equation*}
v_{1}(\xi _{1},\xi _{2})=\mathbf{a}_{1}\,\cdot \mathbf{u}(\mathbf{y}_{1}+\xi
_{2}\mathbf{a}_{2}),
\end{equation*}%
Corollary \ref{cor} gives 
\begin{equation*}
v_{1}^{2}\leq C\sum\limits_{l=1}^{M_{1}(\xi _{1})}\left[ |\gamma \mathbf{u}(%
\mathbf{c}_{1,l}(\xi _{1}))|^{2}+\int_{c_{1,l}(\xi _{1})}^{d_{1,l}(\xi
_{1})}|\mathbb{D}\mathbf{u}(\mathbf{y}_{2}+s\mathbf{a}_{1})|^{2}\ ds\right]
=f_{1}(\xi _{1}).
\end{equation*}%
Multiplying this inequality by \eqref{v2}, integrating over $(\xi _{1},\xi
_{2})\in \Omega $, and proceeding as in the derivation of Equation %
\eqref{e:CrossProductBound}, yields 
\begin{equation*}
\int_{\Omega }v_{2}^{2}v_{1}^{2}\ d\xi _{1}d\xi _{2}\leq C\left\Vert \mathbf{%
u}\right\Vert _{LD_{\gamma }^{2}(\Omega )}^{2}.
\end{equation*}%
Observing that%
\begin{eqnarray*}
\int_{\Omega }v_{2}^{2}v_{1}^{2}\ d\xi _{1}d\xi _{2} &=&\int_{\Omega }(%
\mathbf{a}_{1}\cdot \mathbf{u}(\mathbf{x}))^{2}(\mathbf{a}_{2}\cdot \mathbf{u%
}(\mathbf{x}))^{2}\ d\mathbf{x} \\
&=&\frac{1}{4}\int_{\Omega }(u-v)^{2}(u+v)^{2}\ d\mathbf{x,}
\end{eqnarray*}%
we obtain 
\begin{equation*}
\int_{\Omega }(u^{4}-2u^{2}v^{2}+v^{4})\ d\mathbf{x}\leq C\left\Vert \mathbf{%
u}\right\Vert _{LD_{\gamma }^{2}(\Omega )}^{2}.
\end{equation*}%
Therefore, by combining this estimate with estimate ({\ref%
{e:CrossProductBound})} we conclude that 
\begin{equation*}
\Vert \mathbf{u}\Vert _{L^{4}(\Omega )}^{4}=\int_{\Omega }(u^{4}+v^{4})\ d%
\mathbf{x}\leq C\left\Vert \mathbf{u}\right\Vert _{LD_{\gamma }^{2}(\Omega
)}^{4}.
\end{equation*}%
which coincides with \eqref{e:BDembedding} for $N=2$ and $p=2.$

\bigskip

\textbf{(II)} We now turn to the general $N$-dimensional case$.$ \ We follow
closely the proof of Theorem 6.95, pages 333-336, of \cite{DD}.

We denote the Euclidean basis of $\mathbb{R}^{N}$ by $\{\mathbf{e}%
_{i}\}_{i=1}^{N}$ . \ Given a vector $\mathbf{a}\in \mathbb{S}^{N-1}$ and a
point $\mathbf{x}\in \Omega $, we let $\mathbf{y}=\text{Proj}_{\mathbf{a}}%
\mathbf{x}\in P_{\mathbf{a}}$ be the projection of $\mathbf{x}$ on the plane 
$P_{\mathbf{a}}$, and we let 
\begin{equation*}
\Omega _{\mathbf{a}}(\mathbf{y})=l_{\mathbf{a}}(\mathbf{y})\cap \Omega
\end{equation*}%
be the intersection of $\Omega $ with the line parallel to $\mathbf{a}$ and
crossing $\mathbf{y}$ $($and $\mathbf{x).}$\ 

Since $\Omega $ is, by hypothesis, a set of finite perimeter, for $\mathcal{L%
}^{N-1}-$a.a. $\mathbf{y}=\text{Proj}_{\mathbf{a}}\mathbf{x}\in P_{\mathbf{a}}$,$\
\ \ \Omega _{\mathbf{a}}(\mathbf{y})$ is a finite number $M_{\mathbf{a}}(%
\mathbf{y})$ of open intervals with disjoint closures. Consequently, for $%
\mathcal{L}^{N-1}-$a.a. $\mathbf{y}=\text{Proj}_{\mathbf{a}}\mathbf{x}\in P_{%
\mathbf{a}}$, the point $\mathbf{x}$\ belongs to one of these intervals and
its endpoints, which we denote by $\mathbf{c}_{k}(\mathbf{x})$, $\mathbf{d}%
_{k}(\mathbf{x})$ are on the essential boundary $\partial ^{\ast }\Omega $
of $\Omega $. For simplicity of notation, we assume that this interval is
described as 
\begin{equation*}
(\mathbf{c}_{\mathbf{a},k}(\mathbf{y}),\mathbf{d}_{\mathbf{a},k}(\mathbf{y}%
))=\left\{ \,\mathbf{x}\in \mathbb{R}^{N}:~\mathbf{x}=\mathbf{y}+t\mathbf{a},%
\mathbf{\qquad }t\in (c_{\mathbf{a},k}(\mathbf{y}),d_{\mathbf{a},k}(\mathbf{y%
}))\right\} .
\end{equation*}%
If we consider the function%
\begin{equation*}
v_{\mathbf{a}}(\mathbf{x})=\mathbf{a}\cdot \mathbf{u}(\mathbf{x}%
)=\sum_{i=1}^{N}a_{i}u_{i}(\mathbf{x}),
\end{equation*}%
then Corollary \ref{cor} implies that%
\begin{equation*}
|v_{\mathbf{a}}(\mathbf{x})|\leq |\gamma (v_{\mathbf{a}})(\mathbf{c}_{%
\mathbf{a},k}(\mathbf{y}))|+\int_{c_{\mathbf{a},k}(\mathbf{y})}^{d_{\mathbf{a%
},k}(\mathbf{y})}|\mathbf{a}\mathbb{D}\mathbf{u}(\mathbf{y}+s\mathbf{a}%
)\cdot \mathbf{a|}\,\ ds.
\end{equation*}%
We have the elementary inequalities%
\begin{equation*}
|\gamma (v_{\mathbf{a}})|\leq C|\gamma \mathbf{u}|,\qquad |\mathbf{a}\mathbb{%
D}\mathbf{u}(\mathbf{y}+s\mathbf{a})\cdot \mathbf{a}|\leq C|\mathbb{D}(%
\mathbf{u})|.
\end{equation*}%
Here and below $C$ are constants depending only on $N,$ $p$\ and the
diameter of $\Omega .$ Therefore,
\begin{eqnarray}
|v_{\mathbf{a}}(\mathbf{x})|^{p} &\leq &C\sum_{l=1}^{M_{\mathbf{a}}(\mathbf{y%
})}\left( |\gamma \mathbf{u}(\mathbf{c}_{\mathbf{a},l}(\mathbf{y}%
))|^{p}+\int_{c_{\mathbf{a},l}(\mathbf{y})}^{d_{\mathbf{a},l}(\mathbf{y})}|%
\mathbb{D}(\mathbf{u})(\mathbf{y}+s\,\mathbf{a})|^{p}\ ds\right)  \notag \\
&=&H_{N}(\mathbf{u})(\mathbf{y}),  \label{ineq1}
\end{eqnarray}%
for $\mathcal{L}^{N-1}-$a.a. $\mathbf{y}=\text{Proj}_{\mathbf{a}}\mathbf{x}\in P_{%
\mathbf{a}}.$

We introduce the orthonormal projections%
\begin{equation*}
\mathbf{h}_{k}=\frac{\mathbf{a}-a_{k}\,\mathbf{e}_{k}}{|\mathbf{a}-a_{k}\,%
\mathbf{e}_{k}|}\qquad \text{for each }k=1,2,...,N-1
\end{equation*}%
of the vector $\mathbf{a}$ onto the coordinates hyperplanes, identified
canonically with $\mathbb{R}^{N-1}$. As in \eqref{ineq1}, for a fixed $k\in
\{1,...,N-1\}$ \ the function $v_{\mathbf{h}_{k}}(\mathbf{x})=\mathbf{h}%
_{k}\cdot \mathbf{u}(\mathbf{x})$ satisfies the inequality 
\begin{eqnarray}
|v_{\mathbf{h}_{k}}(\mathbf{x})|^{p} &\leq &C\sum_{l=1}^{M_{\mathbf{h}_{k}}(%
\mathbf{y}^{\prime })}\left( |\gamma \mathbf{u}(\mathbf{c}_{\mathbf{h}%
_{k},l}(\mathbf{y}^{\prime }))|^{p}+\int_{c_{\mathbf{h}_{k},l}(\mathbf{y}%
^{\prime })}^{d_{\mathbf{h}_{k},l}(\mathbf{y}^{\prime })}|\mathbb{D}(\mathbf{%
u})(\mathbf{y}^{\prime }+s\,\mathbf{h}_{k})|^{p}\ ds\right)  \notag \\
&=&I_{k}(\mathbf{u})(\mathbf{y}^{\prime }),  \label{ineq2}
\end{eqnarray}%
for $\mathcal{L}^{N-1}-$a.a. $\mathbf{y}^{\prime }=\text{Proj}_{\mathbf{h}%
_{k}}\mathbf{x}\in P_{\mathbf{h}_{k}}$. Similarly, the function $v_{\mathbf{e%
}_{k}}(\mathbf{x})=\mathbf{e}_{k}\cdot \mathbf{u}(\mathbf{x})$ satisfies 
\begin{eqnarray}
|v_{\mathbf{e}_{k}}(\mathbf{x})|^{p} &\leq &C\sum_{l=1}^{M_{\mathbf{e}_{k}}(%
\mathbf{y}^{\prime \prime })}\left( |\gamma \mathbf{u}(\mathbf{c}_{\mathbf{e}%
_{k},l}(\mathbf{y}^{\prime \prime }))|^{p}+\int_{c_{\mathbf{e}_{k},l}(%
\mathbf{y}^{\prime \prime })}^{d_{\mathbf{e}_{k},l}(\mathbf{y}^{\prime
\prime })}|\mathbb{D}(\mathbf{u})(\mathbf{y}^{\prime \prime }+s\,\mathbf{e}%
_{k})|^{p}\ ds\right)  \notag \\
&=&J_{k}(\mathbf{u})(\mathbf{y}^{\prime \prime }),  \label{ineq3}
\end{eqnarray}%
for $\mathcal{L}^{N-1}-$a.a. $\mathbf{y}^{\prime \prime }=\text{Proj}_{%
\mathbf{e}_{k}}\mathbf{x}\in P_{\mathbf{e}_{k}}.$ \ Keeping $k$ fixed, it
follows that 
\begin{equation*}
v_{\mathbf{a}}(\mathbf{x})=\sum_{i=1}^{N}a_{i}u_{i}(\mathbf{x})=v_{\mathbf{h}%
_{k}}(\mathbf{x})+a_{k}v_{\mathbf{e}_{k}}(\mathbf{x}).
\end{equation*}%
Consequently,%
\begin{equation}
|v_{\mathbf{a}}(\mathbf{x})|^{p}\leq C\left[ I_{k}(\mathbf{u})(\mathbf{y}%
^{\prime })+J_{k}(\mathbf{u})(\mathbf{y}^{\prime \prime })\right] .
\label{e:ModvNuRep2}
\end{equation}%
We next use estimates \eqref{ineq1}-\eqref{ineq3} to bound%
\begin{equation*}
|v_{\mathbf{a}}(\mathbf{x})|^{pN}\leq C\,H_{N}(\mathbf{u})\,\prod_{k=1}^{N-1}%
\left[ I_{k}(\mathbf{u})+J_{k}(\mathbf{u})\right] .
\end{equation*}

Using this inequality and the elementary bound 
\begin{equation*}
\left( \alpha _{1}+...+\alpha _{n}\right) ^{1/(N-1)}\leq n^{1/(N-1)}(\alpha
_{1}^{1/(N-1)}+...+\alpha _{n}^{1/(N-1)}),
\end{equation*}%
which is valid for any positive $\alpha _{1},...,\alpha _{n}$ and for any $%
n\in \mathbb{N}$\ (in particular for $n=2^{N-1}$), one can show that $|v_{%
\mathbf{a}}(\mathbf{x})|^{pN/(N-1)}$ is bounded by a linear combination of $%
2^{N-1}$ terms of the form%
\begin{equation}
I_{\sigma }=\left( H_{1}\ldots H_{N}\right) ^{1/(N-1)},  \label{terms}
\end{equation}%
where $H_{k}$ denotes either $I_{k}$ or $J_{k}$.

Each of the terms $H_{k}$ in the product above depends on $N-1$ variables,
and hence we can apply Proposition \ref{l:productDecomp}. To see this fact,
we introduce an adapted basis $\{\mathbf{E}_{k}\}_{k=1}^{N}$ as follows. For
each index $k\in \{1,\ldots ,N-1\}$, we pick a vector $\mathbf{E}_{k}$
belonging to the set $\left\{ \mathbf{h}_{k},\mathbf{e}_{k}\right\} $, and
for $k=N$, we set $\mathbf{E}_{N}=\mathbf{a}.$ If all components of the
vector $\mathbf{a}$ are non zero, it is then easy to see that 
\begin{equation*}
\{\mathbf{E}_{k}\}_{k=1}^{N-1}\text{\quad is a basis of }\mathbb{R}^{N-1}%
\text{\qquad and\qquad }\{\mathbf{E}_{k}\}_{k=1}^{N}\text{\quad is a basis
of }\mathbb{R}^{N}.
\end{equation*}%
For a proof of this fact we refer to Lemma 6.96, page 334-335, of \cite{DD}.
We let $\xi _{j},$ $j=1,...,N,$\ denote the coordinates of $\mathbf{x}\in 
\mathbb{R}^{N}$ in the basis $\mathbf{E}_{1},...,\mathbf{E}_{N}$, that is, 
\begin{equation*}
\mathbf{x}=\sum_{j=1}^{N}x_{i}\mathbf{e}_{_{j}}=\sum_{j=1}^{N}\xi _{j}\,%
\mathbf{E}_{j}
\end{equation*}%
and identify $\mathbf{x}$ with the vector $\boldsymbol{\xi }%
=\sum_{j=1}^{N}\xi _{j}\,\mathbf{E}_{j}.$

Then, each term $I_{\sigma }$\ can be rewritten as 
\begin{equation*}
(I_{\sigma }(\mathbf{x}(\boldsymbol{\xi })))^{N-1}=\prod_{k=1}^{N}\theta
_{k}(\widehat{\mathbf{\xi }}_{k}),  \; \quad \theta _{k}(%
\widehat{\boldsymbol{\xi }}_{k})=H_{k}(\text{Proj}_{\mathbf{E}_{k}}\mathbf{x}(%
\boldsymbol{\xi })), \; k=1,...,N.
\end{equation*}
Proceeding as in the derivation of \eqref{e:CrossProductBound} gives 
\begin{equation*}
\int_{\mathbb{R}^{N-1}}\theta _{k}(\widehat{\boldsymbol{\xi }}_{k})\,d%
\widehat{\boldsymbol{\xi }}_{k}\leq \left\Vert \mathbf{u}\right\Vert
_{LD_{\gamma }^{p}(\Omega )}^{p}.
\end{equation*}%
By Proposition \ref{l:productDecomp} it follows that 
\begin{equation}
\int_{\Omega }I_{\sigma }(\mathbf{x})\,d\mathbf{x}\leq C(\sigma
)\,\prod_{i=1}^{N}\left( \int_{\mathbb{R}^{N-1}}\theta _{k}\,d\widehat{%
\boldsymbol{\xi }}_{k}\right) ^{\frac{1}{N-1}}\leq C\,\left\Vert \mathbf{u}%
\right\Vert _{LD_{\gamma }^{p}(\Omega )}^{\frac{pN}{N-1}},  \label{e:NdimEst}
\end{equation}%
where the dependence on $\sigma $ in the constant $C$ comes from the
Jacobian of the change of variables from $\mathbf{x}$ to $\boldsymbol{\xi }$%
. Next, the integration over $\Omega $ \ of $\ |v_{\mathbf{a}}(\mathbf{x}%
)|^{pN/(N-1)},$ which is a linear combination of $2^{N-1}$ terms of the form
given in \eqref{terms}, yields%
\begin{equation*}
\int_{\Omega }|v_{\mathbf{a}}(\mathbf{x})|^{pN/(N-1)}\,d\mathbf{x}\leq
C\,\left\Vert \mathbf{u}\right\Vert _{LD_{\gamma }^{2}(\Omega )}^{\frac{pN}{%
N-1}}.
\end{equation*}

Lastly, we observe that, since $\mathbf{a}$ can be chosen arbitrarily away
from the coordinate planes, by varying $\mathbf{a}$ we can bound $\Vert
u_{i}\Vert _{L^{pN/(N-1)}}$ for each component $u_{i}$ of $\mathbf{u}$ as
exemplified in the two-dimensional case. For example, choosing $\mathbf{a}=%
\frac{1}{\sqrt{N}}(1,\ldots ,1)$ first and the $\Bar{\mathbf{a}}=\frac{1}{%
\sqrt{N}}(1,\ldots ,-1,\dots ,1),$ where $-1$ is in the $i$-th component,
gives a bound on%
\begin{equation*}
\Vert u_{i}\Vert _{L^{pN/(N-1)}(\Omega )}=\frac{\sqrt{N}}{2}\Vert v_{\mathbf{%
a}}-v_{\Bar{\mathbf{a}}}\Vert _{L^{pN/(N-1)}(\Omega )}.
\end{equation*}%
We conclude that estimate \eqref{e:BDembedding} holds.
\end{proof}

\bigskip

\begin{remark}
We make some final observations. The embedding of Theorem \ref{t:maintheorem}
is an analog of the embedding\ $W_{0}^{1,p}(\Omega )\hookrightarrow
L^{q}(\Omega )$ \ for $q=\frac{pN}{(N-p)},$ which is valid for arbitrary
open set (see e.g. \cite[Theorem 4.1.1., page 177]{Zi}). By comparison, we
allow for non-zero trace values at the boundary, but we require minimum
regularity on the boundary of the domain (i.e., finite perimeter) in order
to define and control the trace. Indeed, Theorem \ref{t:maintheorem} shows
that elements of the space $\,LD_{\gamma }^{p}(\Omega )\,$ have less
integrability than those in $W_{0}^{1,p}(\Omega )\,$, as it is expected
because Korn's inequality does not generally hold on domains with finite
perimeter (see, for example, \cite{Acosta3} for domains with cusps).
\end{remark}

\end{document}